\documentclass{article}
\usepackage{amssymb}
\usepackage{latexsym}
\usepackage{amsmath}
\usepackage{amsthm,pstricks}
\topskip  \theoremstyle{remark}
\theoremstyle{plain}
\newtheorem{theorem}{Theorem}[section]
\newtheorem{lemma}[theorem]{Lemma}

\newtheorem{proposition}[theorem]{Proposition}

\newtheorem{corollary}[theorem]{Corollary}

\newcommand{\asc}{\text{asc}}
\newcommand{\zero}{\text{zeros}}
\newcommand{\fwd}{\text{fwd}}
\newcommand{\RLm}{\text{RLmax}}
\begin{document}
%-----------------------------------------------------------
\title{Some enumerative results related to ascent sequences}
\author{Toufik Mansour\\
\small Department of Mathematics, University of Haifa, 31905 Haifa, Israel\\[-0.8ex]
\small\texttt{tmansour@univ.haifa.ac.il}\\[1.8ex]
Mark Shattuck\\
\small Department of Mathematics, University of Haifa, 31905 Haifa, Israel\\[-0.8ex]
\small\texttt{shattuck@math.utk.edu, maarkons@excite.com}\\[1.8ex]
}

\date{\small }%\today}
\maketitle
%-----------------------------------------------------------

\begin{abstract}
An \emph{ascent sequence} is one consisting of non-negative integers in which the size of each letter is restricted by the number of ascents preceding it in the sequence.  Ascent sequences have recently been shown to be related to (2+2)-free posets and a variety of other combinatorial structures.  In this paper, we prove in the affirmative some recent conjectures concerning pattern avoidance for ascent sequences.  Given a pattern $\tau$, let $\mathcal{S}_\tau(n)$ denote the set of ascent sequences of length $n$ avoiding $\tau$.  Here, we show that the joint distribution of the statistic pair $(\asc,\zero)$ on $\mathcal{S}_{0012}(n)$ is the same as $(\asc,\RLm)$ on the set of $132$-avoiding permutations of length $n$.  In particular, the ascent statistic on $\mathcal{S}_{0012}(n)$ has the Narayana distribution.  We also enumerate $S_\tau(n)$ when $\tau=1012$ and $\tau=0123$ and confirm the conjectured formulas in these cases.  We combine combinatorial and algebraic techniques 
 to prove our results, in two cases, making use of the kernel method.  Finally, we discuss the case of avoiding $210$ and determine two related recurrences.
\end{abstract}



\section{Introduction}

An \emph{ascent} in a sequence $x_1x_2\cdots x_k$ is a place $j \geq 1$ such that $x_j<x_{j+1}$.  An \emph{ascent sequence} $x_1x_2\cdots x_n$ is one consisting of non-negative integers satisfying $x_1=0$ and for all $i$ with $1<i\leq n$,
$$x_i\leq \asc(x_1x_2\cdots x_{i-1})+1,$$
where $\asc(x_1x_2\cdots x_k)$ is the number of ascents in the sequence $x_1x_2\cdots x_k$.  An example of such a sequence is $01013212524$, whereas $01003221$ is not, because $3$ exceeds $\asc(0100)+1=2$.  Starting with the paper by Bousquet-M{\'e}lou, Claesson, Dukes, and Kitaev \cite{BCD}, where they were related to the (2+2)-free posets and the generating function was determined, ascent sequences have since been studied in a series of papers where connections to many other combinatorial structures have been made.  See, for example, \cite{DP,DRS,KR} as well as \cite[Section 3.2.2]{K} for further information.

In this paper, we answer some recent conjectures in the affirmative which were raised by Duncan and Steingr\'{\i}msson \cite{DS} concerning the avoidance of patterns by ascent sequences.  The patterns considered are analogous to patterns considered originally on permutations and later on other structures such as $k$-ary words and finite set partitions.

By a \emph{pattern}, we will mean a sequence of non-negative integers, where repetitions are allowed.  Let $\pi=\pi_1\pi_2\cdots \pi_n$ be an ascent sequence and $\tau=\tau_1\tau_2\cdots\tau_m$ be a pattern.  We will say that $\pi$ \emph{contains} $\tau$ if $\pi$ has a subsequence that is order isomorphic to $\tau$, that is, there is a subsequence $\pi_{f(1)},\pi_{f(2)},\ldots,\pi_{f(m)}$, where $1\leq f(1)<f(2)<\cdots<f(m)\leq n$, such that for all $1 \leq i,j \leq m$, we have $\pi_{f(i)}<\pi_{f(j)}$ if and only if $\tau_i<\tau_j$ and $\pi_{f(i)}>\pi_{f(j)}$ if and only if $\tau_i>\tau_j$. Otherwise, the ascent sequence $\pi$ is said to \emph{avoid} the pattern $\tau$.  For example, the ascent sequence $0120311252$ has three occurrences of the pattern $100$, namely, the subsequences $211$, $311$, and $322$, but avoids the pattern $210$.  Note that within an occurrence of a pattern $\tau$, letters corresponding to equal letters in $\tau$ must be equal within the occurrence.

Following \cite{DS}, we will write patterns for ascent sequences using non-negative integers, though patterns for other structures like permutations have traditionally been written with positive integers, to be consistent with the usual notation for ascent sequences which contains $0$'s.  Thus, the traditional patterns will have different names here; for example $123$ becomes $012$ and $221$ becomes $110$.

If $\tau$ is a pattern, then let $\mathcal{S}_\tau(n)$ denote the set of ascent sequences of length $n$ that avoid $\tau$ and $A_\tau(n)$ the number of such sequences. The set of $\emph{right-to-left maxima}$ in a sequence of numbers $a_1a_2\cdots a_n$ is the set of $a_i$ such that $a_i>a_j$ for all $j>i$.  Let RLmax($x$) be the number of right-to-left maxima in a sequence $x$.  Recall that the Catalan numbers are given by $C_n=\frac{1}{n+1}\binom{2n}{n}$ and that the Narayana numbers given by $N_{n,k}=\frac{1}{n}\binom{n}{k}\binom{n}{k-1}$, $1 \leq k \leq n$, refine the Catalan numbers in that $C_n=\sum_{k=1}^n N_{n,k}$.  It is well known that the number of $132$-avoiding permutations of length $n$ having exactly $k$ ascents is given by $N_{n,k+1}$.

Let $\fwd(x)$ be the length of the maximal final weakly decreasing sequence in an ascent sequence $x$.  For example, $\fwd(010013014364332)=5$ since $64332$ has length $5$.  We also let $\zero(x)$ denote the number of $0$'s in an ascent sequence $x$.

We now state here the conjectures from \cite{DS} which we prove in the affirmative in the following sections.

\textbf{Conjecture 3.2.} \emph{We have $A_{0012}(n)=C_n$, the $n$-th Catalan number. Moreover, the bistatistic $(\asc,\fwd)$ on $\mathcal{S}_{0012}(n)$ has the same distribution as $(\asc,RLmax)$ does on permutations avoiding the pattern $132$.  In particular, this implies that the number of ascents has the Narayana distribution on $\mathcal{S}_{0012}(n)$.  Also, the bistatistics $(\asc,\fwd)$ and $(\asc,\zero)$ have the same distribution on $0012$-avoiding ascent sequences.}

\textbf{Conjecture 3.4.} \emph{The number $A_{0123}(n)$ equals the number of Dyck paths of semilength $n$ and height at most $5$.  See sequence A080937 in \cite{Sl}.}

See Corollary \ref{c1} and Theorems \ref{t2} and \ref{t4} below.  We also show half of the following conjecture; see Theorem \ref{t3} below.

\textbf{Conjecture 3.5.} \emph{The patterns $0021$ and $1012$ are Wilf equivalent, and $A_{0021}(n)=A_{1012}(n)$ is given by the binomial transform of the Catalan numbers, which is sequence A007317 in \cite{Sl}.}

We remark that in our proof of Conjecture 3.2 in the next section, we first consider the joint distribution $(\asc,\zero,\fwd)$ on the members of $\mathcal{S}_{0012}(n)$ not ending in $0$ and determine a functional equation satisfied by its generating function which we denote by $g(x,y;u,v)$.  Using the \emph{kernel method} \cite{BBD}, we are then able to show $g(x,y;1,u)=g(x,y;u,1)$, which implies the equidistribution of $(\asc,\fwd)$ and $(\asc,\zero)$ on $\mathcal{S}_{0012}(n)$.
Comparison with the generating function for the distribution of $(\asc,\RLm)$ on $132$-avoiding permutations then gives the first part of Conjecture 3.2. Furthermore, an expression for the generating function $g(x,y;u,v)$ may be recovered and the full distribution for $(\asc,\zero,\fwd)$ can be obtained by extracting the coefficient of $x^n$ from it.

In the third section, we enumerate $S_\tau(n)$ when $\tau=1012$ and $\tau=0123$, in the former case, making use of the kernel method.  In this proof, we first describe refinements of the numbers $A_{1012}(n)$ by introducing appropriate auxiliary statistics on $\mathcal{S}_{1012}(n)$ and then write recurrences for these refined numbers.  The recurrences may then be expressed as a functional equation which may be solved using the kernel method.  See \cite{Z} for a further description and examples of this strategy of refinement in determining an explicit formula for a sequence. We conclude with a discussion of the case of avoiding $210$ by ascent sequences.  We determine two related recurrences using combinatorial arguments which perhaps may shed some light on Conjecture 3.3 in \cite{DS}.

\section{Distribution of some statistics on $\mathcal{S}_{0012}(n)$}

In order to determine the distributions of some statistics on $\mathcal{S}_{0012}(n)$, we first refine the set as follows.  Given $n \geq 1$, $0 \leq m \leq n-1$, and $1 \leq r,l \leq n-m$, let $A_{n,m,r,\ell}$ denote the subset of $\mathcal{S}_{0012}(n)$ whose members have $m$ ascents, $r$ zeros, and $\fwd$ value $\ell$.  For example, $\pi=012334004332 \in A_{12,5,3,4}$.  Let $a_{n,m,r,\ell}=|A_{n,m,r,\ell}|$; note that $\sum_{m,r,\ell}a_{n,m,r,\ell}=A_{0012}(n)$ for all $n$.

In what follows, it will be more convenient to deal with the members of $\mathcal{S}_{0012}(n)$ that do not end in a $0$.  Let $B_{n,m,r,\ell}$ denote the subset of $A_{n,m,r,\ell}$ whose members do not end in a $0$ and let $b_{n,m,r,\ell}=|B_{n,m,r,\ell}|$.  The array $b_{n,m,r,\ell}$ may be determined as described in the following lemma.

\begin{lemma}\label{l1}
The array $b_{n,m,r,\ell}$ may assume non-zero values only when $n \geq 2$, $1 \leq m \leq n-1$, $1 \leq r \leq n-m$, and $1 \leq \ell \leq n-m$.  It is determined for $n \geq 3$ by the recurrences
\begin{equation}\label{l1e1}
b_{n,m,1,\ell}=\sum_{i=1}^{n-m}\sum_{j=0}^{t} b_{n-j-1,m-1,i-j,\ell-j}, \qquad m \geq 2,
\end{equation}
where $t=\min\{i-1,\ell-1\}$, and
\begin{equation}\label{l1e2}
b_{n,m,r,\ell}=b_{n-1,m,r-1,\ell}+\sum_{j=\ell+1}^{n-m} b_{n-1,m-1,r-1,j}, \qquad m \geq 2 \text{ and } r \geq 2,
\end{equation}
and by the condition
\begin{equation}\label{l1e3}
b_{n,1,r,\ell}=\begin{cases} 1, & \text{if} \text{~}\text{~} r+\ell=n; \\ 0, & otherwise, \end{cases} \end{equation}
if $n \geq 2$.
\end{lemma}
\begin{proof}
The first statement is clear from the definitions. If $m=1$, then the set $B_{n,1,r,\ell}$ is either empty or is a singleton consisting of a sequence of the form $0^r1^{n-r}$, $1 \leq r \leq n-1$, which implies \eqref{l1e3}.   If $m \geq 2$ and $r=1$, then $\pi \in B_{n,m,1,\ell}$ must be of the form $\pi=1\pi'$, where $\pi' \in A_{n-1,m-1,i,\ell}$ for some $1 \leq i \leq n-m$.  For each $i$, note that
$$|A_{n-1,m-1,i,\ell}|=\sum_{j=0}^t b_{n-j-1,m-1,i-j,\ell-j},$$
upon conditioning on the number of trailing zeros $j$ within a member of $A_{n-1,m-1,i,\ell}$.  Summing over $i$ gives \eqref{l1e1}.

For \eqref{l1e2}, we condition on the position of the right-most zero.  First observe that within $\pi \in B_{n,m,r,\ell}$, where $r \geq 2$, the right-most zero must directly precede the first letter in the final weakly decreasing sequence, i.e., it is the lower number in the right-most ascent.  (For if not, then there would be an occurrence of $0012$, with the ``$1$'' and ``$2$'' corresponding to the letters in the right-most ascent.)  One may then obtain a particular member of $B_{n,m,r,\ell}$ by inserting a zero directly before the $\ell$-th letter from the right within some member of $B_{n-1,m,r-1,\ell}$ or by inserting a zero just before the $j$-th letter from the right within some member of $B_{n-1,m-1,r-1,j}$ for some $j \in \{\ell+1,\ell+2,\ldots,n-m\}$.  Note that no additional ascent is created in the former case, while in the latter, an ascent is introduced since a zero has been inserted between two numbers $a$ and $b$, where $a \geq b \geq 1$.  Summing over $j$ g
 ives \eqref{l1e2}.
\end{proof}

If $n \geq 2$ and $1 \leq m \leq n-1$, then let
$$B_{n,m,r}(u)=\sum_{r=1}^{n-m} b_{n,m,r,\ell}u^\ell, \qquad 1 \leq r \leq n-m.$$
Let $$B_{n,m}(u,v)=\sum_{r=1}^{n-m}B_{n,m,r}(u)v^r, \qquad 1 \leq m \leq n-1.$$

The polynomials $B_{n,m}(u,v)$ satisfy the following recurrence.

\begin{lemma}\label{l2}
If $n \geq 3$ and $2 \leq m \leq n-1$, then
\begin{align}
B_{n,m}(u,v)&=vB_{n-1,m}(u,v)+\frac{v}{1-u}(uB_{n-1,m-1}(1,v)-B_{n-1,m-1}(u,v))\notag\\
&~~+v\sum_{j=0}^{n-m-1}u^jB_{n-j-1,m-1}(u,1)\label{l2e1},
\end{align}
with
\begin{align}
B_{n,1}(u,v)&=\frac{uv(u^{n-1}-v^{n-1})}{u-v}.\label{l2e2}
\end{align}
\end{lemma}
\begin{proof}
First observe that if $n \geq 2$ and $m=1$, then $B_{n,1,r}=u^{n-r}$ so that
$$B_{n,1}(u,v)=\sum_{r=1}^{n-1}u^{n-r}v^r=\frac{uv(u^{n-1}-v^{n-1})}{u-v}.$$
If $m \geq 2$ and $r \geq 2$, then we have, by Lemma \ref{l1},
\begin{align}
B_{n,m,r}(u)&=B_{n-1,m,r-1}(u)+\sum_{\ell=1}^{n-m}u^\ell \sum_{j=\ell+1}^{n-m}b_{n-1,m-1,r-1,j}\notag\\
&=B_{n-1,m,r-1}(u)+\sum_{j=2}^{n-m}b_{n-1,m-1,r-1,j}\sum_{\ell=1}^{j-1}u^\ell\notag\\
&=B_{n-1,m,r-1}(u)+\frac{1}{1-u}\sum_{j=2}^{n-m}b_{n-1,m-1,r-1,j}(u-u^j)\notag\\
&=B_{n-1,m,r-1}(u)+\frac{1}{1-u}\left(uB_{n-1,m-1,r-1}(1)-ub_{n-1,m-1,r-1,1}\right)\notag\\
&~~-\frac{1}{1-u}\left(B_{n-1,m-1,r-1}(u)-ub_{n-1,m-1,r-1,1}\right)\notag\\
&=B_{n-1,m,r-1}(u)+\frac{1}{1-u}\left(uB_{n-1,m-1,r-1}(1)-B_{n-1,m-1,r-1}(u)\right),\label{l2e3}
\end{align}
with
\begin{align}
B_{n,m,1}(u)&=\sum_{\ell=1}^{n-m}u^\ell\sum_{i=1}^{n-m}\sum_{j=0}^{\ell-1}b_{n-j-1,m-1,i-j,\ell-j}\notag\\
&=\sum_{i=1}^{n-m}~\sum_{j=0}^{n-m-1}u^j \sum_{\ell=j+1}^{n-m}b_{n-j-1,m-1,i-j,\ell-j}u^{\ell-j}\notag\\
&=\sum_{i=1}^{n-m}~\sum_{j=0}^{n-m-1}u^jB_{n-j-1,m-1,i-j}(u)\notag\\
&=\sum_{j=0}^{n-m-1}u^j\sum_{i=j+1}^{n-m}B_{n-j-1,m-1,i-j}(u).\label{l2e4}
\end{align}
Multiplying \eqref{l2e3} by $v^r$, summing over $2 \leq r \leq n-m$, and adding $v$ times equation \eqref{l2e4}
gives
\begin{align*}
B_{n,m}(u,v)&=vB_{n-1,m}(u,v)+\frac{v}{1-u}(uB_{n-1,m-1}(1,v)-B_{n-1,m-1}(u,v))\\
&~~+v\sum_{j=0}^{n-m-1}u^jB_{n-j-1,m-1}(u,1),
\end{align*}
which completes the proof.
\end{proof}

If $n \geq 2$, then let $$B_n(y;u,v)=\sum_{m=1}^{n-1} B_{n,m}(u,v)y^m.$$  Let
$$g(x,y;u,v)=\sum_{n\geq 2}B_n(y;u,v)x^n$$
denote the generating function for the sequence $B_n(y;u,v)$.  Then $g$ satisfies the following functional equation.

\begin{lemma}\label{l3}
We have
\begin{equation}\label{l3e1}
\left(1-vx+\frac{vxy}{1-u}\right)g(x,y;u,v)=\frac{uvx^2y}{1-ux}+\frac{uvxy}{1-u}g(x,y;1,v)+\frac{vxy}{1-ux}g(x,y;u,1).
\end{equation}
\end{lemma}
\begin{proof}
If $n \geq 3$, then by \eqref{l2e1} and \eqref{l2e2}, we have
\begin{align}
B_n(y;u,v)&=\frac{uvy(u^{n-1}-v^{n-1})}{u-v}+v\left(B_{n-1}(y;u,v)-\frac{uvy(u^{n-2}-v^{n-2})}{u-v}\right)\notag\\
&~~+\frac{vy}{1-u}(uB_{n-1}(y;1,v)-B_{n-1}(y;u,v))+vy\sum_{j=0}^{n-3}u^j\sum_{m=2}^{n-j-1}B_{n-j-1,m-1}(u,1)y^{m-1}\notag\\
&=u^{n-1}vy+vB_{n-1}(y;u,v)+\frac{vy}{1-u}(uB_{n-1}(y;1,v)-B_{n-1}(y;u,v))\notag\\
&~~+vy\sum_{j=0}^{n-3}u^{n-j-3}B_{j+2}(y;u,1).\label{l3e2}
\end{align}
Since $B_2(y;u,v)=uvy$, equation \eqref{l3e2} is also seen to hold when $n=2$, provided we define $B_1(y;u,v)=0$.
Multiplying \eqref{l3e2} by $x^n$, and summing over $n \geq 2$, implies
\begin{align*}
g(x,y;u,v)&=\sum_{n\geq 2}u^{n-1}vx^{n}y+vxg(x,y;u,v)+\frac{vxy}{1-u}\left(ug(x,y;1,v)-g(x,y;u,v)\right)\\
&~~+vy\sum_{j \geq 0}\frac{B_{j+2}(y;u,1)}{u^{j+3}}\sum_{n \geq j+3}(ux)^n\\
&=\frac{uvx^2y}{1-ux}+vxg(x,y;u,v)+\frac{vxy}{1-u}\left(ug(x,y;1,v)-g(x,y;u,v)\right)\\
&~~+\frac{vxy}{1-ux}g(x,y;u,1),
\end{align*}
which yields \eqref{l3e1}.
\end{proof}

\begin{theorem}\label{t1}
We have
\begin{equation}\label{t1e1}
g(x,y;u,1)=g(x,y;1,u)=\frac{uxy(1-ux)\kappa-u^2x^2y}{(1-u)(1-ux)+uxy},
\end{equation}
where $\kappa=\kappa(x,y)$ is given by
\begin{equation}\label{t1e1a}
\kappa=\frac{1-x(y+1)-\sqrt{(1-x(y+1))^2-4x^2y}}{2xy}.
\end{equation}
\end{theorem}
\begin{proof}
Letting $v=1$ in \eqref{l3e1} implies
\begin{equation}\label{t1e2}
\left(1-x+\frac{xy}{1-u}-\frac{xy}{1-ux}\right)g(x,y;u,1)=\frac{ux^2y}{1-ux}+\frac{uxy}{1-u}g(x,y;1,1).
\end{equation}
To solve \eqref{t1e2}, we will use the \emph{kernel method} (see \cite{BBD}).  Let $u_o=u_o(x,y)$ satisfy
$$1-x+\frac{xy}{1-u_o}-\frac{xy}{1-u_ox}=0,$$
i.e.,
$$u_o=\frac{1+x(1-y)-\sqrt{(1+x(1-y))^2-4x}}{2x}.$$
Letting $u=u_o$ in \eqref{t1e2}, and solving for $g(x,y;1,1)$, then gives
$$g(x,y;1,1)=-\frac{x(1-u_o)}{1-u_ox}=(1-x)\kappa-x,$$
where the second equality follows from comparing $x(u_o-1)$ with $(1-u_ox)((1-x)\kappa-x)$ after simplifying.  (Note that there were two possible values for $u_o$ and our choice was dictated by the condition $g(0,y;1,1)=0$.)
Thus,
\begin{align*}
g(x,y;u,1)&=\frac{ux^2y(1-u)+uxy(1-ux)g(x,y;1,1)}{(1-x)((1-u)(1-ux)+uxy)}\\
&=\frac{ux^2y(1-u)+uxy(1-ux)((1-x)\kappa-x)}{(1-x)((1-u)(1-ux)+uxy)}\\
&=\frac{uxy(1-ux)\kappa-u^2x^2y}{(1-u)(1-ux)+uxy},
\end{align*}
which gives half of \eqref{t1e1}.

To find an expression for $g(x,y;1,u)$, we again use the kernel method.  Let $v_o=v_o(x,y,u)$ satisfy
$$1-v_ox+\frac{v_oxy}{1-u}=0,$$
i.e., $v_o=\frac{1-u}{x(1-u-y)}$.
Substituting $v=v_o$ in \eqref{l3e1} implies
\begin{equation}\label{t1e3}
g(x,y;1,v_o)=\frac{u-1}{uv_oxy}\left(\frac{uv_ox^2y}{1-ux}+\frac{v_oxy}{1-ux}g(x,y;u,1)\right).
\end{equation}
Letting $v_o=w$ in \eqref{t1e3} then gives
\begin{equation}\label{t1e4}
g(x,y;1,w)=\frac{u-1}{uwxy}\left(\frac{uwx^2y}{1-ux}+\frac{wxy}{1-ux}g(x,y;u,1)\right),
\end{equation}
where $u=\frac{1-wx(1-y)}{1-wx}$.  Substituting into \eqref{t1e4} the expression determined above for $g(x,y;u,1)$, and simplifying, implies after several algebraic steps,
$$g(x,y;1,w)=\frac{wxy(1-wx)\kappa-w^2x^2y}{(1-w)(1-wx)+wxy},$$
which completes the proof.
\end{proof}

Note that the full expression for $g(x,y;u,v)$ can now be recovered from \eqref{t1e1} and \eqref{l3e1}.  Let $\mathcal{B}_n$ denote the subset of $\mathcal{S}_{0012}(n)$ whose members do not end in $0$.  Taking $u=y=1$ in \eqref{t1e1} shows that there are $C_{n}-C_{n-1}$ members of $\mathcal{B}_n$ if $n \geq 1$ and thus $C_n$ members of $\mathcal{S}_{0012}(n)$ altogether. Let $f(x,y;u)$ be the generating function counting the members of $\mathcal{S}_{0012}(n)$ according to the number of ascents and the length of the final weakly decreasing sequence.

\begin{corollary}\label{c1}
The bistatistics (\asc,\fwd) and (\asc,\zero) have the same distribution on $\mathcal{S}_{0012}(n)$.  Furthermore, the common generating function $f(x,y;u)$ has explicit formula
\begin{equation}\label{c1e1}
f(x,y;u)=\frac{1}{1-ux}+\frac{1}{1-ux}g(x,y;1,u),
\end{equation}
where $g(x,y;1,u)$ is given by \eqref{t1e1}.
\end{corollary}
\begin{proof}
Theorem \ref{t1} implies that the bistatistics (asc,fwd) and (asc,zeros) are equally distributed on $\mathcal{B}_n$ for all $n$.  Since adding an arbitrary number of trailing zeros to a member of $\mathcal{B}_n$ preserves the number of ascents while increasing the length of the final weakly decreasing sequence and the number of zeros by the same amount, it follows that (asc,fwd) and (asc,zeros) are also equally distributed on $\mathcal{S}_{0012}(n)$ for all $n$.  Furthermore, note that a member of $\mathcal{S}_{0012}(n)$ having at least one ascent may be obtained by adding $i$ zeros for some $i$ to the end of some member of $\mathcal{B}_{n-i}$.  Each added zero increases the $\fwd$ value by one, which justifies the $\frac{1}{1-ux}$ factor in the second term on the right-hand side of \eqref{c1e1}.  The $\frac{1}{1-ux}$ term  counts all ascent sequences having no ascents, i.e., those of the form $0^n$ for some $n \geq 0$.
\end{proof}

The following result answers the remaining part of Conjecture 3.2 above in the affirmative.

\begin{theorem}\label{t2}
The bistatistic $(\asc,\fwd)$ on $\mathcal{S}_{0012}(n)$ has the same distribution as $(asc,RLmax)$ on the set of $132$-avoiding permutations of length $n$.  In particular, the number of ascents has the Narayana distribution on $\mathcal{S}_{0012}(n)$.
\end{theorem}
\begin{proof}
The second statement is an immediate consequence of the first since it is well known that $\asc$ has the Narayana distribution on $132$-avoiding permutations.  To show the first statement, let $h(x,y;u)$ denote the generating function which counts the $132$-avoiding permutations of length $n$ according to the number of ascents and the number of right-to-left maxima.  We will show
\begin{equation}\label{t2e1}
h(x,y;u)=f(x,y;u).
\end{equation}
We first compute $h(x,y;u)$  Considering whether or not $n$ is the first letter of a non-empty $132$-avoiding permutation of length $n$ implies
\begin{equation}\label{t2e2}
h(x,y;u)=1+uxh(x,y;u)+uxy(h(x,y;1)-1)h(x,y;u),
\end{equation}
which gives
$$h(x,y;u)=\frac{1}{1-ux(1-y)-uxyh(x,y;1)}.$$
Taking $u=1$ in \eqref{t2e2} and solving for $h(x,y;1)$ implies $h(x,y;1)=\kappa+1$, where $\kappa=\kappa(x,y)$ is defined by \eqref{t1e1a} above, which gives
\begin{equation}\label{t2e3}
h(x,y;u)=\frac{1}{1-ux-uxy\kappa}.
\end{equation}

On the other hand, by \eqref{c1e1} and \eqref{t1e1}, we have
\begin{align}
f(x,y;u)&=\frac{1}{1-ux}+\frac{1}{1-ux}g(x,y;1,u)\notag\\
&=\frac{1}{1-ux}+\frac{1}{1-ux}\left(\frac{uxy(1-ux)\kappa-u^2x^2y}{(1-u)(1-ux)+uxy}\right)\notag\\
&=\frac{1-u+uxy(\kappa+1)}{(1-u)(1-ux)+uxy}.\label{t2e4}
\end{align}

Equality \eqref{t2e1} now follows from \eqref{t2e3} and \eqref{t2e4}, upon verifying
$$\frac{1}{1-ux-uxy\kappa}=\frac{1-u+uxy(\kappa+1)}{(1-u)(1-ux)+uxy},$$
which may be done by cross-multiplying, expanding both sides of the equation that results, and using the relation $xy\kappa^2=(1-x(y+1))\kappa-x$.
\end{proof}

\section{Other patterns}

In this section, we consider the problem of determining $A_\tau(n)$ in the cases when $\tau=1012$, $0123$, or $210$.  We will use the following additional notation.  If $n$ is a positive integer, then let $[n]=\{1,2,\ldots,n\}$, with $[0]=\varnothing$.  If $m$ and $n$ are positive integers, then let $[m,n]=\{m,m+1,\ldots,n\}$ if $m \leq n$, with $[m,n]=\varnothing$ if $m>n$.

\subsection{The case 1012}

Here, we enumerate the members of $\mathcal{S}_{1012}(n)$.  Recall that a sequence $\pi=\pi_1\pi_2\cdots \pi_n$ is said to be a \emph{restricted growth function} (RGF) if it satisfies (i) $\pi_1=1$ and (ii) $\pi_{i+1}\leq \max\{\pi_1,\pi_2,\ldots,\pi_i\}+1$ for all $i\in[n-1]$.  See, e.g., \cite{M} for details. By \cite[Lemma 2.4]{DS}, the set $\mathcal{S}_{1012}(n)$ consists solely of RGF sequences since $1012$ is a subpattern of $01012$.  So we consider the avoidance problem on RGF's, or, equivalently, on finite set partitions.

Recall that a \emph{partition} of $[n]$ is any collection of non-empty, pairwise disjoint subsets, called \emph{blocks}, whose union is $[n]$.  A partition $\Pi$ is said to be in \emph{standard form} if it is written as $\Pi=B_1/B_2/\cdots$, where $\min(B_1)<\min(B_2)<\cdots$.  One may also represent $\Pi$, equivalently, by the \emph{canonical sequential form} $\pi=\pi_1\pi_2\cdots \pi_n$, wherein $j \in B_{\pi_j}$ for each $j$; see, e.g., \cite{SW} for details.  For example, the partition $\Pi=1,3,6/2,4/5,8/7$ has canonical sequential form $\pi=12123143$.  Note that $\pi$ is a restricted growth function from $[n]$ onto $[k]$, where $k$ denotes the number of blocks of $\Pi$.  Below, we will represent partitions $\Pi$ by their canonical sequential forms $\pi$ and consider an avoidance problem on these words.   See, for example, the related papers \cite{CDD,JM,Sa} concerning the problem of pattern avoidance on set partitions.

Note than an RGF, equivalently, a set partition, avoids the pattern $1012$ if and only if it avoids $01012$.  In what follows, we will denote $01012$ by $12123$ to be consistent with the convention of RGF's starting with the letter $1$.  We now address the problem of avoiding $12123$.  Let $P_n$ denote the set of all partitions of $[n]$ and let $P_n(12123)$ consist of those members of $P_n$ that avoid the pattern $12123$ when represented canonically.

We refine the set $P_n(12123)$ as follows.  Given $n \geq 2$ and $1 \leq s < t \leq n$, let $A_{n,t,s}$ denote the subset of $P_n(12123)$ consisting of those partitions $\pi=\pi_1\pi_2\cdots\pi_n$ having at least two distinct letters in which the left-most occurrence of the largest letter is at position $t$ and the left-most occurrence of the second largest letter is at position $s$.  For example, $\pi=123324425215 \in A_{12,9,6}$ since the left-most occurrence of the largest letter, namely, $5$, is at position $9$ and the left-most occurrence of the second largest letter is at position $6$.  The array $a_{n,t,s}=|A_{n,t,s}|$ is determined by the following recurrence.

\begin{lemma}\label{l4}
The array $a_{n,t,s}$ can assume non-zero values only when $n \geq 2$ and $1 \leq s<t\leq n$.  It is determined by the recurrence
\begin{equation}\label{l4e1}
a_{n,t,s}=\sum_{j=t}^{n-1}a_{n-1,j,s}+\sum_{r=1}^{t-s}\sum_{i=1}^{s-1}a_{n-r,t-r,i}, \qquad n \geq 3 ~\text{and}~2\leq s<t \leq n,
\end{equation}
and the condition
\begin{equation}\label{l4e2}
a_{n,t,1}=2^{n-t}, \qquad n \geq 2 ~\text{and}~2 \leq t\leq n.
\end{equation}
\end{lemma}
\begin{proof}
Note that members of $A_{n,t,1}$ are of the form $1^{t-1}2\alpha$, where $\alpha$ is any word on the letters $\{1,2\}$, which implies \eqref{l4e2}.  To show \eqref{l4e1}, first suppose that $\pi=\pi_1\pi_2\cdots \pi_n \in A_{n,t,s}$, where $n \geq 3$ and $2 \leq s < t \leq n$.  Let us denote the largest letter of $\pi$ by $z$.  We first show that there are $\sum_{j=t}^{n-1}a_{n-1,j,s}$ members $\pi$ of $A_{n,t,s}$ in which $z$ occurs at least twice.  To do so, first note that if $\pi$ contains two or more letters $z$, then the $z$ at position $t$ (i.e., the left-most $z$) is extraneous concerning the avoidance of $12123$ since all letters coming to the left of it are also governed by a $z$ to the right of position $t$.  Thus, we may safely delete the $z$ at position $t$ and the left-most occurrence of $z$ in the resulting partition of $[n-1]$ is at position $j$ for some $j \in [t,n-1]$; note that the left-most position of the second largest letter remains unchanged.  Thus, de
 letion of the left-most $z$ defines a bijection between the subset of $A_{n,t,s}$ in which $z$ occurs at least twice and $\bigcup_{j=t}^{n-1}A_{n-1,j,s}$, which has cardinality $\sum_{j=t}^{n-1}a_{n-1,j,s}$.

So it remains to show that there are $\sum_{r=1}^{t-s}\sum_{i=1}^{s-1}a_{n-r,t-r,i}$ members $\pi$ of $A_{n,t,s}$ in which $z$ occurs once.  Suppose that the $z-1$ at position $s$ within such $\pi$ is the first letter in a run of $(z-1)$'s of length $r$.  Then $1 \leq r \leq t-s$ and no other $(z-1)$'s may occur between positions $s$ and $t$ without introducing an occurrence of $12123$ (note $s \geq 2$ implies $z \geq 3$).  Thus, we may delete the run of $(z-1)$'s starting at position $s$ since all letters to the left of position $s$ are also governed by the $z$ at position $t$. We then change the $z$ at position $t$ to a $z-1$.  For each $r$, this change of letter and deletion defines a bijection with $\bigcup_{i=1}^{s-1}A_{n-r,t-r,i}$, which has cardinality $\sum_{i=1}^{s-1}a_{n-r,t-r,i}$.  Summing over $r$ then implies \eqref{l4e1} and completes the proof.
\end{proof}

Our next result shows in the affirmative half of Conjecture 3.5 above.

\begin{theorem}\label{t3}
We have
\begin{equation}\label{t3e1}
A_{1012}(n)=\sum_{i=0}^{n-1}\binom{n-1}{i}C_i, \qquad n \geq 1,
\end{equation}
where $C_i$ denotes the $i$-th Catalan number.
\end{theorem}
\begin{proof}
We determine a generating function for the sum $\sum_{t=2}^n\sum_{s=1}^{t-1}a_{n,t,s}$, where $n \geq 2$.  To do so, we first
define the polynomials $A_{n,t}(v)=\sum_{s=1}^{t-1}a_{n,t,s}v^{s-1}$ and $A_{n}(u,v)=\sum_{t=2}^nA_{n,t}(v)u^{t-2}$. Note that $A_{n,t}(0)=a_{n,t,1}=2^{n-t}$ for all $2 \leq t\leq n$, which implies $A_n(u,0)=\frac{2^{n-1}-u^{n-1}}{2-u}$. Multiplying \eqref{l4e1} by $u^{t-2}v^{s-1}$, and summing over $s=2,3,\ldots,t-1$ and $t=3,4,\ldots,n$, gives
\begin{align*}
A_{n}(u,v)-A_n(u,0)&=\sum_{t=3}^n\sum_{s=2}^{t-1}\sum_{j=t}^{n-1}a_{n-1,j,s}v^{s-1}u^{t-2}+\sum_{t=3}^n\sum_{s=2}^{t-1}\sum_{r=1}^{t-s}\sum_{i=1}^{s-1}a_{n-r,t-r,i}v^{s-1}u^{t-2}\\
&=\sum_{s=2}^{n-2}\sum_{j=s+1}^{n-1}\sum_{t=s+1}^ja_{n-1,j,s}v^{s-1}u^{t-2}+\sum_{s=2}^{n-1}\sum_{r=1}^{n-s}\sum_{t=s+r}^{n}\sum_{i=1}^{s-1}a_{n-r,t-r,i}v^{s-1}u^{t-2}\\
&=\sum_{s=2}^{n-2}\sum_{j=s+1}^{n-1}a_{n-1,j,s}v^{s-1}\left(\frac{u^{s-1}-u^{j-1}}{1-u}\right)\\
&~~+\sum_{s=2}^{n-1}\sum_{r=1}^{n-s}\sum_{t=s}^{n-r}\sum_{i=1}^{s-1}a_{n-r,t,i}v^{s-1}u^{r+t-2}\\
&=\sum_{j=3}^{n-1}\sum_{s=2}^{j-1}a_{n-1,j,s}v^{s-1}\left(\frac{u^{s-1}-u^{j-1}}{1-u}\right)\\
&~~+\sum_{r=1}^{n-2}\sum_{t=2}^{n-r}\sum_{i=1}^{t-1}a_{n-r,t,i}u^{t+r-2}\left(\frac{v^i-v^t}{1-v}\right)\\
\end{align*}
\begin{align*}
&\qquad~=\frac{1}{1-u}(A_{n-1}(1,uv)-uA_{n-1}(u,v)-A_{n-1}(1,0)+uA_{n-1}(u,0))\\
&\qquad~~~+\sum_{j=2}^{n-1}\frac{u^{n-j}}{1-v}(vA_j(u,v)-v^2A_j(uv,1)), \qquad n \geq 3,
\end{align*}
which implies
\begin{align*}
A_{n}(u,v)&=\frac{2^{n-1}-u^{n-1}}{2-u}+\frac{1}{1-u}(A_{n-1}(1,uv)-uA_{n-1}(u,v)-2^{n-2}+1+\frac{u}{2-u}(2^{n-2}-u^{n-2}))\\
&+\sum_{j=2}^{n-1}\frac{u^{n-j}}{1-v}(vA_j(u,v)-v^2A_j(uv,1)), \qquad n \geq 3,
\end{align*}
which is also seen to hold when $n=2$ upon taking $A_1(u,v)=0$.

Next define the generating function
$$A(t;u,v)=\sum_{n\geq1}A_n(u,v)t^n.$$
Multiplying the last recurrence relation by $t^n$, and summing over $n\geq2$, we obtain
\begin{align*}
A(t;u,v)&=\frac{t^2}{(1-ut)(1-t)}+\frac{t}{1-u}(A(t;1,uv)-uA(t;u,v))\\
&~~+\frac{uvt}{(1-v)(1-ut)}(A(t;u,v)-vA(t;uv,1)).
\end{align*}
Substituting $u=1/v$ into the last equation then yields
\begin{align*}
\left(1+\frac{t}{v-1}-\frac{vt}{(1-v)(v-t)}\right)A(t;1/v,v)&=\frac{vt^2}{(1-t)(v-t)}+\frac{vt}{v-1}A(t;1,1)\\
&~~-\frac{v^2t}{(1-v)(v-t)}A(t;1,1),
\end{align*}
which is equivalent to
\begin{align}\label{t3e2}
\left(1-\frac{t(2v-t)}{(1-v)(v-t)}\right)A(t;1/v,v)&=\frac{vt^2}{(1-t)(v-t)}-\frac{vt(2v-t)}{(1-v)(v-t)}A(t;1,1).
\end{align}
To solve \eqref{t3e2}, we use the kernel method. If we set the coefficient of $A(t;1/v,v)$ in \eqref{t3e2} equal to zero, and solve for $v=v_o$ in terms of $t$, we obtain
$$v_o=\frac{1-t+\sqrt{1-6t+5t^2}}{2}.$$
Substituting $v=v_o$ into \eqref{t3e2} then gives
$$A(t;1,1)=\frac{t(1-v_o)}{(1-t)(2v_o-t)}=\frac{1-3t-\sqrt{1-6t+5t^2}}{2(1-t)}.$$
(Note that there were two possible values for $v_o$, and our choice for $v_o$ was dictated by the condition $A(0;1,1)=0$.)
Thus the generating function for the sequence $A_{1012}(n)=1+\sum_{t=2}^n\sum_{s=1}^{t-1}a_{n,t,s}$, $n\geq1$, is given by
$$\frac{t}{1-t}+A(t;1,1)=\frac{1-t-\sqrt{1-6t+5t^2}}{2(1-t)}=\frac{1}{2}\left(1-\sqrt{\frac{1-5t}{1-t}}\right).$$
It is easily verified that this is also the generating function for the Catalan transform sequence $\sum_{i=0}^{n-1}\binom{n-1}{i}C_i$, $n \geq 1$, which completes the proof.
\end{proof}

\subsection{The case 0123}

Here, we determine $A_{0123}(n)$ and answer Conjecture 3.4 above in the affirmative.

\begin{theorem}\label{t4}
Let $a_n=A_{0123}(n)$.  Then $a_n$ is given by the recurrence
\begin{equation}\label{t4e0}
a_n=5a_{n-1}-6a_{n-2}+a_{n-3}, \qquad n \geq 3,
\end{equation}
with $a_0=a_1=1$ and $a_2=2$.
\end{theorem}
\begin{proof}
Let $b_n$ denote the number of $0123$-avoiding ascent sequences having at least three distinct letters. Then $a_n=b_n+2^{n-1}$ if $n \geq 1$, with $a_0=1$, upon including all possible binary sequences.  We first find an explicit formula for $b_n$. To do so, suppose that an ascent sequence $\pi$ enumerated by $b_n$ has largest letter $\ell+1$, where the left-most occurrence of $\ell+1$ corresponds to the larger number in the $r$-th ascent for some $r$.  Since $\pi$ avoids $0123$, the only letters occurring prior to $\ell+1$ are $0$'s and $1$'s.  Since exactly $r-1$ ascents involving $0$ and $1$ occur prior to the left-most $\ell+1$, we must have $2r-2\leq n-1$ so that $2 \leq r \leq \frac{n+1}{2}$. By definition of an ascent sequence, we have $\ell+1 \leq r$ so that $1 \leq \ell \leq r-1$.  Then $\pi$ must be of the form
$$\pi=\alpha(\ell+1)\beta,$$
where $\alpha$ is a binary sequence starting with $0$ having exactly $r-1$ ascents for some $2 \leq r \leq \frac{n+1}{2}$, $\ell+1$ is the largest letter for some $1 \leq \ell \leq r-1$, and $\beta$ is possibly empty.

Suppose further that $|\alpha|=i$ and that the total length of the first $2r-2$ runs of letters is $j$, where $2r-2 \leq i \leq n-1$ and $2r-2 \leq j \leq i$.  Then there are $\binom{j-1}{2r-3}$ choices for the first $j$ letters of $\alpha$, with the last $i-j$ letters of $\alpha$ being $0$.  We now turn our consideration to the subword $\beta$, which is of length $n-i-1$.  The subsequence $L$ comprising all letters of $\beta$ from the set $[2,\ell+1]$ must be non-increasing, with $L$ possibly empty or comprising all of $\beta$.  Furthermore, there is no restriction within $\beta$ concerning the relative positions of $0$'s and $1$'s.  Let $s=|L|$, where $0 \leq s \leq n-i-1$.  Then there are $n-i-s-1$ positions to be occupied by $0$'s and $1$'s within $\beta$, and thus there are
$$\binom{n-i-1}{n-i-s-1}2^{n-i-s-1}=\binom{n-i-1}{s}2^{n-i-s-1}$$
possibilities for these letters.  Once the choices for the positions of the $0$'s and $1$'s within $\beta$ have been made, there are $\binom{s+\ell-1}{\ell-1}$ choices for the subsequence $L$ since it is of length $s$ with its letters in non-increasing order coming from the set $[2,\ell+1]$.  Summing over all possible $r$, $\ell$, $i$, $j$, and $s$, we obtain
\begin{equation}\label{t4e1}
b_n=\sum_{r=2}^{\frac{n+1}{2}}\sum_{\ell=1}^{r-1}\sum_{i=2r-2}^{n-1}\sum_{j=2r-2}^i \sum_{s=0}^{n-i-1} \binom{j-1}{2r-3}\binom{n-i-1}{s}\binom{s+\ell-1}{\ell-1}2^{n-i-s-1}, \qquad n \geq 1.
\end{equation}

We now compute the generating function $\sum_{n \geq 1}b_nx^n$.  By \eqref{t4e1}, we have
\begin{align*}
\sum_{n \geq 1}&b_nx^n=\sum_{n\geq 1}x^n\left(\sum_{r=2}^{\frac{n+1}{2}}\sum_{\ell=1}^{r-1}\sum_{i=2r-2}^{n-1}\sum_{j=2r-2}^i \sum_{s=0}^{n-i-1} \binom{j-1}{2r-3}\binom{n-i-1}{s}\binom{s+\ell-1}{\ell-1}2^{n-i-s-1}\right)\\
&=\sum_{\ell \geq 1}\sum_{r \geq \ell+1} \sum_{j \geq 2r-2} \sum_{i \geq j} \sum_{s \geq 0} \sum_{n\geq i+s+1}\binom{j-1}{2r-3}\binom{n-i-1}{s}\binom{s+\ell-1}{\ell-1}2^{n-i-s-1}x^n\\
\end{align*}
\begin{align*}
&=\sum_{\ell \geq 1}\sum_{r \geq \ell+1} \sum_{j \geq 2r-2} \sum_{i \geq j}\sum_{s \geq0} \binom{j-1}{2r-3}\binom{s+\ell-1}{\ell-1}2^{-s}x^{i+1}\sum_{n \geq i+s+1}\binom{n-i-1}{s}(2x)^{n-i-1}\\
&=\frac{1}{1-2x}\sum_{\ell \geq 1}\sum_{r \geq \ell+1} \sum_{j \geq 2r-2} \sum_{i \geq j}\binom{j-1}{2r-3}x^{i+1}\sum_{s \geq 0} \binom{s+\ell-1}{\ell-1}\left(\frac{x}{1-2x}\right)^s\\
&=\frac{1}{1-2x}\sum_{\ell \geq 1}\sum_{r \geq \ell+1} \sum_{j \geq 2r-2} \sum_{i \geq j}\binom{j-1}{2r-3}x^{i+1}\left(\frac{1-2x}{1-3x}\right)^{\ell},\\
\end{align*}
where we have used the fact $\sum_{n \geq i}\binom{n}{i}x^n=\frac{x^i}{(1-x)^{i+1}}$.  Rearranging factors in the last sum implies
\begin{align*}
\sum_{n\geq1}b_nx^n&=\frac{1}{1-2x}\sum_{\ell \geq 1}\left(\frac{1-2x}{1-3x}\right)^{\ell} \sum_{r \geq \ell+1}\sum_{j \geq 2r-2} \binom{j-1}{2r-3}\sum_{i \geq j}x^{i+1}\\
&=\frac{x^2}{(1-x)(1-2x)}\sum_{\ell \geq 1}\left(\frac{1-2x}{1-3x}\right)^{\ell} \sum_{r \geq \ell+1}\sum_{j \geq 2r-2} \binom{j-1}{2r-3}x^{j-1}\\
&=\frac{x}{(1-x)(1-2x)}\sum_{\ell \geq 1}\left(\frac{1-2x}{1-3x}\right)^{\ell} \sum_{r \geq \ell+1}\left(\frac{x}{1-x}\right)^{2r-2}\\
&=\frac{x(1-x)}{(1-2x)^2}\sum_{\ell \geq 1}\left(\frac{1-2x}{1-3x}\right)^{\ell}\cdot\left(\frac{x}{1-x}\right)^{2\ell}\\
&=\frac{x(1-x)}{(1-2x)^2}\cdot \frac{x^2(1-2x)}{(1-x)^2(1-3x)}\cdot \frac{1}{1-\frac{x^2(1-2x)}{(1-x)^2(1-3x)}}\\
&=\frac{x^3(1-x)}{(1-2x)(1-5x+6x^2-x^3)}.
\end{align*}

Then we have
\begin{align*}
\sum_{n \geq 0}a_nx^n&=1+\sum_{n\geq 1}2^{n-1}x^n+\sum_{n\geq 1}b_nx^n\\
&=\frac{1-x}{1-2x}+\frac{x^3(1-x)}{(1-2x)(1-5x+6x^2-x^3)}\\
&=\frac{(1-x)(1-3x)}{1-5x+6x^2-x^3},
\end{align*}
which implies $A_{0123}(n)$ is given by \eqref{t4e0}.
\end{proof}

Thus $A_{0123}(n)$ coincides with sequence A080937 in \cite{Sl}, which also counts the Dyck paths of semilength $n$ and height at most $5$, and it would be interesting to determine a direct bijection.

\subsection{Some remarks on the case $210$}

Duncan and Steingr\'{\i}msson \cite{DS} made the following conjecture concerning the avoidance of $210$ by ascent sequences:

\textbf{Conjecture 3.3.} \emph{The number $A_{210}(n)$ equals the number of non-3-crossing set partitions of $\{1,2,\ldots,n\}$.  See sequence A108304 in \cite{Sl}.}

While we were unable to enumerate the members of $\mathcal{S}_{210}(n)$ and confirm Conjecture 3.3, we did determine some combinatorial structure in this case.  Perhaps the recurrence in either of the propositions below  would be a first step in proving this conjecture, once the proper technique is applied.  Given $m \geq 0$ and $0 \leq s \leq r \leq m$, let $C_{n,m,r,s}$ denote the subset of $\mathcal{S}_{210}(n)$ whose members have exactly $m$ ascents, largest letter $r$, and last letter $s$.  For example, $\pi=012330115223 \in C_{12,6,5,3}$. Let $f_{m,r,s}=f_{m,r,s}(x)$ be the generating function (g.f.) which counts the members of $C_{n,m,r,s}$, where $m$, $r$, and $s$ are fixed.  Note that $f_m(x)=\sum_{r,s}f_{m,r,s}(x)$ is the g.f. counting the members of $\mathcal{S}_{210}(n)$ having exactly $m$ ascents and $f(x)=1+\sum_{m\geq 0}f_m(x)$ is the g.f., counting all members of $\mathcal{S}_{210}(n)$.

The following proposition provides a recurrence for the $f_{m,r,s}$.

\begin{proposition}\label{p1}
The array of generating functions $f_{m,r,s}$, where $m \geq 0$ and $0\leq s \leq r \leq m$, is determined by the initial condition $f_{0,0,0}=\frac{x}{1-x}$, and for $m \geq 1$, the recurrences
\begin{align}
f_{m,r,s}&=\frac{x}{1-x}\sum_{i=0}^{s-1}f_{m-1,r,i}+\frac{x}{1-x}\sum_{j=s+1}^{r-1}f_{m-1,j,s}+\frac{x^2}{(1-x)^2}\sum_{i=0}^s\sum_{j=i}^r f_{m-1,j,i}\notag\\
&~~-\frac{x^2}{(1-x)^2}\sum_{j=s+1}^{r-1}\sum_{i=0}^{s-1}f_{m-2,j,i}, \qquad 0 \leq s < r \leq m, \label{p1e1}
\end{align}
and
\begin{equation}\label{p1e2}
f_{m,r,r}=\frac{x}{1-x}\sum_{j=0}^{r}\sum_{i=0}^j f_{m-1,j,i}-\frac{x}{1-x}f_{m-1,r,r}, \qquad 1 \leq r \leq m.
\end{equation}
\end{proposition}
\begin{proof}
The initial condition is clear, the enumerated sequences being those consisting of all zeros.  To show \eqref{p1e1}, first observe that an ascent sequence enumerated by $f_{m,r,s}$ in this case must have at least three distinct runs of letters, with the last run of the letter $s$.  We condition on the letter occurring in the next-to-last run.  Note that this letter is either $i$ for some $0 \leq i \leq s-1$ or is $r$, for otherwise there would be an occurrence of $210$.  Then the term $\frac{x}{1-x}\sum_{i=0}^{r-1}f_{m-1,r,i}$ is the g.f. counting all members of $C_{n,m,r,s}$ in the first case since the $s$'s occurring in the final run are seen to be extraneous concerning a possible occurrence of $210$ and thus may be removed without loss of structure (whence the factor $\frac{x}{1-x}$ which accounts for these letters).

Now assume the next-to-last run is of the letter $r$.  We differentiate two further cases where the second-to-last run is (i) of the letter $i$ for some $0 \leq i \leq s$ or (ii) of the letter $j$ for some $j \in [s+1,r-1]$.  In case (i), one may simply delete the two final runs of letters, noting that an ascent has been removed in so doing, which implies the g.f. in this case is given by $\frac{x^2}{(1-x)^2}\sum_{i=0}^s\sum_{j=i}^sf_{m-1,j,i}$.

In case (ii), we may delete the final run of the letter $r$, which removes an ascent but does not otherwise affect the structure. The resulting ascent sequences are then enumerated by $\frac{x}{1-x}\sum_{j=s+1}^{r-1}f_{m-1,j,s}^*$, where $f_{m-1,j,s}^*$ is the g.f. which counts the same ascent sequences as $f_{m-1,j,s}$ but with the added condition that the next-to-last run is of the letter $j$.  Observe that
$$f_{m-1,j,s}^*=f_{m-1,j,s}-\frac{x}{1-x}\sum_{i=0}^{s-1}f_{m-2,j,i},$$
upon subtracting the g.f. for those sequences whose next-to-last run is of a letter smaller than $s$.  Combining the three cases above gives \eqref{p1e1}.

For \eqref{p1e2}, considering separately the cases in which the letter $r$ occurs in (i) exactly one run or in (ii) two or more runs yields
\begin{align*}
f_{m,r,r}&=\frac{x}{1-x}\sum_{j=0}^{r-1}\sum_{i=0}^j f_{m-1,j,i}+\frac{x}{1-x}\sum_{i=0}^{r-1}f_{m-1,r,i}\\
&=\frac{x}{1-x}\sum_{j=0}^{r}\sum_{i=0}^j f_{m-1,j,i}-\frac{x}{1-x}f_{m-1,r,r},
\end{align*}
which completes the proof.
\end{proof}

Note that $f_{m,0,0}=0$ for all $m \geq 1$ and that $f_{m,r,s}$ is always of the form
$\frac{p(x)}{(1-x)^{2m+1}}$ for some polynomial $p(x)$ of degree
at most $2m+1$. Using the recurrences, one gets for $m=1$,
$$f_{1,1,0}=\frac{x^3}{(1-x)^3} \text{~and~}
f_{1,1,1}=\frac{x^2}{(1-x)^2},$$
which implies
$f_1(x)=\frac{x^2}{(1-x)^3}$.

When $m=2$, one gets
$$f_{2,1,0}=\frac{x^5}{(1-x)^5},~f_{2,1,1}=\frac{x^4}{(1-x)^4},~f_{2,2,0}=\frac{x^4}{(1-x)^5}=f_{2,2,1} \text{~and~}
f_{2,2,2}=\frac{x^3}{(1-x)^4},$$
which implies $f_2(x)=\frac{x^3(1+2x)}{(1-x)^5}$.

Since members of $\mathcal{S}_{210}(n)$ seem to be equinumerous with the partitions of $[n]$ having no $3$-crossings, perhaps the techniques used in \cite{BX} to enumerate the latter could be applied to the former in the absence of an obvious bijection between the two structures.

One may also determine recurrences similar to those in Lemmas \ref{l1} and \ref{l4} above.  Let $C_{n,m,r,s}$ be as defined above  and let $D_{n,m,r,s}$ consist of those members of $C_{n,m,r,s}$ whose next-to-last letter is $r$.  We have the following recurrences for the numbers $c_{n,m,r,s}=|C_{n,m,r,s}|$ and $d_{n,m,r,s}=|D_{n,m,r,s}|$.

\begin{proposition}\label{p2}
The arrays $c_{n,m,r,s}$ and $d_{n,m,r,s}$ may assume non-zero values only when $n \geq 1$ and $0 \leq s \leq r \leq m < n$.  For $n \geq 2$, they satisfy the recurrences
\begin{equation}\label{p2e1}
c_{n,m,r,s}=c_{n-1,m,r,s}+d_{n,m,r,s}+\sum_{i=0}^{s-1}c_{n-1,m-1,r,i}, \qquad 0 \leq s < r \leq m,
\end{equation}
\begin{equation}\label{p2e2}
d_{n,m,r,s}=d_{n-1,m,r,s}+\sum_{i=s+1}^{r-1}d_{n-1,m-1,i,s}+\sum_{i=0}^s\sum_{j=i}^r c_{n-2,m-1,j,i}, \qquad 0 \leq s<r \leq m,
\end{equation}
\begin{equation}\label{p2e3}
c_{n,m,r,r}=d_{n,m,r,r}+\sum_{i=0}^{r-1}\sum_{j=i}^rc_{n-1,m-1,j,i}, \qquad 0 \leq r \leq m,
\end{equation}
and
\begin{equation}\label{p2e4}
d_{n,m,r,r}=c_{n-1,m,r,r}, \qquad 0 \leq r \leq m,
\end{equation}
with the initial conditions $c_{1,0,0,0}=1$ and $d_{1,0,0,0}=0$.
\end{proposition}
\begin{proof}
The first statement is clear from the definitions, as are the initial conditions.  Considering whether the penultimate letter within a member of $C_{n,m,r,s}$ is $s$, $r$, or $i$, where $0 \leq i \leq s-1$, yields \eqref{p2e1}.  Note that in the last case, an ascent is lost when the final $s$ is removed, but not in the first case.  Considering whether the antepenultimate letter within a member of $D_{n,m,r,s}$ is $r$ or $i \in \{s+1,s+2,\ldots,r-1\}$ or $i \in\{0,1,\ldots,s\}$ yields \eqref{p2e2}.  Note that the removal of the right-most $r$ in the second case yields a member of $D_{n-1,m-1,i,s}$, while the removal of the right-most $r$ and $s$ in the third case results in a member of $C_{n-2,m-1,j,i}$ for some $j \geq i$.  For \eqref{p2e3}, consider whether the penultimate letter within a member of $C_{n,m,r,r}$ is $r$ or $i\in\{0,1,\ldots,r-1\}$.  In the latter case, we remove the final $r$ and the resulting sequence belongs to $C_{n-1,m-1,j,i}$ for some $j \geq i$.  Finall
 y, for \eqref{p2e4}, note that members of $D_{n,m,r,r}$ are synonymous with members of $C_{n-1,m,r,r}$ upon removing the final $r$.
\end{proof}

Perhaps applying the technique of the second section above to the recurrences in the last proposition, and modifying it somewhat, might enable one to determine the sequence $A_{210}(n)$.

%------------------------------------------------------------

\end{document}